\documentclass{article}

\usepackage{amsmath,amsthm}
\usepackage{amssymb}
\usepackage[total={5in, 8in}]{geometry}
\usepackage{comment}
\usepackage{amsmath,amsthm,amsfonts,amssymb,bm,mathtools,mathrsfs}
\usepackage{enumitem}
\usepackage[hidelinks]{hyperref}
\usepackage{cleveref,caption}
\usepackage{multirow}
\usepackage{tikz,tikz-cd}
\usepackage{quiver}

\newtheorem{Proposition}{Proposition}[section]
\newtheorem{proposition}[Proposition]{Proposition}
\newtheorem{lemma}[Proposition]{Lemma}

\newtheorem{theorem}[Proposition]{Theorem}
\newtheorem{claim}[Proposition]{Claim}
\newtheorem{corollary}[Proposition]{Corollary}
\theoremstyle{definition}

\newtheorem{definition}[Proposition]{Definition}
\newtheorem{question}[Proposition]{Question}

\newtheorem{remark}[Proposition]{Remark}
\crefname{chapter}{Chapter}{Chapters}
\crefname{lemma}{Lemma}{Lemmas}
\crefname{theorem}{Theorem}{Theorems}
\crefname{definition}{Definition}{Definitions}
\crefname{proposition}{Proposition}{Propositions}
\crefname{notation}{Notation}{}
\crefname{corollary}{Corollary}{Corollaries}
\crefname{section}{Section}{Sections}
\crefname{remark}{Proposition}{Remark}

\newcommand{\sset}{\subseteq}
\newcommand{\rset}{\supseteq}

\newcommand{\mc}{\mathcal}

\renewcommand{\phi}{\varphi}
\newcommand{\resim}{\mathord{\sim}}
\newcommand{\upset}[1]{\mathord{\uparrow \! #1}}
\newcommand{\dnset}[1]{\mathord{\downarrow \! #1}}

\newcommand{\alphad}[1]{\phi_{#1}}

\title{Constructive Quantum Logics}
\author{Juan P. Aguilera\\
TU Wien \and
Guillaume Massas\\
SNS Pisa}
\date{\today}

\begin{document}
\maketitle
\begin{abstract}
Following a suggestion of Birkhoff and Von Neumann [\textit{Ann. Math.} 37 (1936), 23-32], we pursue a joint study of quantum logic and intuitionistic logic.
We exhibit a linear-time translation which for each quantum logic $Q$ and each superintuitionistic logic $I$ yields an axiomatization of $Q\cap I$ from axiomatizations of $Q$ and $I$. The translation is centered around a certain axiom \textup{(Ex)} which (together with introduction and elimination rules for connectives) is shown to 
axiomatize the intersection of orthologic and intuitionistic logic, solving a problem of Holliday [\textit{Logics} 1 (2023), pp. 36-79]. 
We prove that the lattice of all super-Ex logics is isomorphic to the product of the lattices of quantum logics and superintuitionistic logics.
We prove that there are infinitely many sub-Ex logics extending Holliday's fundamental logic.
\end{abstract}

\setcounter{tocdepth}{1}
\tableofcontents

\section{Introduction}
The motivating question for this work is: \textit{What are the jointly valid logical principles of quantum mechanics and constructivism?} We shall address this question and develop a theory of constructive quantum logic.
The relevance of constructivism for computer science is well documented. For instance, it is central in the theory of programming languages (see e.g., Curry \cite{Cu34} or Wadler \cite{Wa14}); it is the underlying principle behind the calculus of constructions (see Coquant and Huet \cite{CH88}), on which many proof assistants such as Coq are based; and occurs in the context of applications to artificial intelligence and other domains.
Its importance is also well documented in the domains of philosophy (e.g., in the works of Brouwer, Hilbert, Bernays, Shanin, Markov, Kolmogorov, Bishop, Feferman, and others; see e.g., Bridges and Richman \cite{BR87}) and mathematics (e.g., in the context of applied proof theory; see Kohlenbach \cite{Ko08}).
The archetypal constructivist logic is intuitionistic logic, in which the principle of excluded middle $p \vee \lnot p$ is dropped, although a variety of alternatives exist, such as \textit{De Morgan logic}, in which one accepts the principle of \textit{weak excluded middle} $\lnot p \vee \lnot\lnot p$; see Gabbay and Olivetti \cite{GO00} for background.
The situation with quantum mechanics and its logic is slightly more involved.

The study of quantum logic began
almost 90 years ago with the work of Birkhoff and Von Neumann \cite{BvN36}, motivated by the observation that the collection of observables in quantum mechanics does not obey classical logic. For instance, the principle of distributivity 
\begin{equation}\label{eqDistributivity}
a \wedge (b \vee c) \vdash (a \wedge b) \vee (a \wedge c)
\end{equation}
is not valid: if $a$ is a proposition relating to the momentum of a particle and $b$ and $c$ are propositions concerning the location of a particle, it is easy to see how failures of \eqref{eqDistributivity} are consequences of Heisenberg's uncertainty principle. Similarly, the following pseudo-complement law:
\begin{equation}\label{eqPseudocomp}
    a \land b \vdash \bot \Rightarrow a \vdash \neg b
\end{equation}
is not valid in quantum mechanics.

Nowadays, the term ``quantum logic'' is used to mean slightly different things by different authors. A natural notion of ``quantum logic'' is the first-order theory of closed subspaces of complex Hilbert spaces in the signature $(\vee, \wedge, \lnot, 0, 1)$, where $\lnot$ is the orthogonality relation. It was recently shown by Fritz \cite{Fr20} that this theory is undecidable in the general case. An interesting aspect of Fritz's proof is its use of Slofstra's \cite{Sl20} solution to Tsirelson's problem.
This undecidability result perhaps explains \textit{a posteriori} why quantum logic in this strict sense has not had much application in physics or elsewhere. Nonetheless, there have been proposals for the use of quantum logic to model quantum computations (see Rawling and Selesnick \cite{RS00} or Ying \cite{Yi05}) and it is not unreasonable to expect future applications of this and other quantum logics in this direction as quantum technology gains more relevance. One expects the relevance of constructivism for computer science to translate into applications of constructive quantum logics to quantum computation.

In contrast to Fritz \cite{Fr20}, Dunn et al.~\cite{Detal05} have pointed out that Tarski's quantifier-elimination theorem implies that the quantum logics associated to fixed finite-dimensional Hilbert spaces are decidable. Herrman and Ziegler \cite{HZ16} have determined the computational complexity of quantum satisfiability in each fixed finite dimension. See also Herrman \cite{He22} and Herrman and Ziegler \cite{HZ18} for related decidability results.

Typically quantum logics lack a deduction theorem, and indeed any reasonable notion of material implication. Nonetheless, restricting to conjunction, disjunction, and negation results in  reasonable logics. Moreover, the expressive power of the logics can be enhanced by considering non-trivial expressions of entailment ``$\phi\vdash\psi$'' instead of only one-sided formulas ``$\vdash\psi$.''
Takeuti \cite{Ta81} has explored how quantum logic can be applied to develop an alternate foundation of mathematics. He has famously referred to classical logic as ``the logic of God,'' to intuitionistic logic as ``the logic of people,'' and to quantum logic as ``the logic of things.'' If one subscribes to this view, the motivating question in the first paragraph can be rephrased as \textit{What is the common logic of people and the physical world?}

One of the most common usages of the term ``quantum logic'' is to refer to what some call \textit{minimal quantum logic} or \textit{orthologic} (see e.g., Dishkant \cite{Di72}). This is the logic of ortholattices, of which the algebras of closed subspaces of complex Hilbert spaces form a subcollection.
Orthologic still has the property that the addition of distributivity \eqref{eqDistributivity} results in classical logic, but is much more tractable, and indeed the propositional validity problem is decidable in polynomial time (see Goldblatt \cite{Go74} or Nishimura \cite{Ni94}). In fact, this decidability result extends to first-order orthologic in a certain sense (see \cite{AB24}). The decidability of orthologic has led to applications in a variety of contexts such as automated theorem proving (see Guilloud, Gambir, and Kun\u{c}ak \cite{GGK23}). This work in turn has motivated the study of orthologic with axioms (see Guilloud and Kun\u{c}ak \cite{GK24}). It is interesting that both intuitionistic logic and orthologic have found applications in the context of automated theorem proving, and thus one expects that constructive quantum logic will do so too.

A notable strengthening of orthologic which is also sometimes called ``quantum logic'' is \textit{orthomodular logic}, also considered by Birkhoff and Von Neumann \cite{BvN36}. This is the logic of orthomodular lattices: ortholattices satisfying the \textit{orthomodular law},  the following weakening of distributivity:
\[\text{if $a\leq b$, then $a\vee (\lnot a \wedge b) = b$}.\]

Noting the clear parallel with the situation of intuitionism,
Birkhoff and Von Neumann \cite[\S 17]{BvN36} suggested comparing quantum logic with intuitionistic logic. 
This comparison has been pursued in a variety of contexts. Battilotti and Sambin \cite{BS99} have studied what they call \textit{basic logic}, a natural logic validating the joint principles of orthologic, intuitionistic logic, and Girard's \textit{linear logic}. One of the motivating questions asked in \cite{BS99} was to identify $X$ in Figure \ref{FigureFindX}. 

\begin{figure}[h]
\begin{center}
\begin{tikzcd}
                                       & \text{Classical logic}  &                              \\
\text{Intuitionistic logic} \arrow[ru] &                         & \text{Orthologic} \arrow[lu] \\
                                       & X \arrow[lu] \arrow[ru] &                             
\end{tikzcd}
\end{center}
\caption{Orthologic, intuitionistic logic, and their common principles.}\label{FigureFindX}
\end{figure}
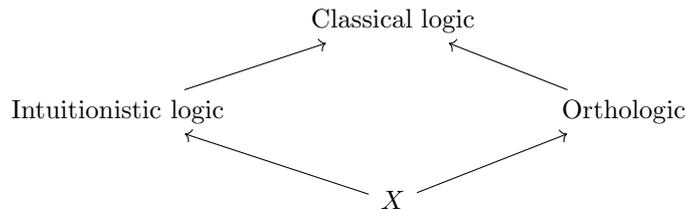

Indeed, Battilotti and Sambin \cite{BS99} give a possible (proof-theoretic) answer to this question which essentially consists of adding to their basic logic principles which identify the two types of conjunctions, disjunctions, and top and bottom elements of linear logics. This produces a natural logic which they call BS and which can be naturally extended to intuitionistic logic and orthologic by added their characteristic structural rules. Aoyama \cite{Ao04}  gave a similar proof-theoretic answer to this question via a system he called DO (dual orthologic). Selesnick \cite{Se03} studied a similar logic called \text{intuitionistic orthologic}.	

Holliday \cite{Ho23} considered Figure \ref{FigureFindX} from a semantic perspective
and the logics in it and developed a theory of what is now called \textit{fundamental logic}. His study was motivated by earlier work of his and Mandelkern \cite{HM24} on the logical principles underpinning natural language in the presence of epistemic modalities. Fundamental logic is a very natural logic and deserving of its name, in the authors' opinion. Like intuitionistic logic, it lacks the validity of double negation elimination; like orthologic, it lacks the validity of distributivity \eqref{eqDistributivity}. Moreover, fundamental logic becomes intuitionistic logic when augmented with distributivity and the pseudo-complement law, and it becomes orthologic when augmented with double negation elimination. We shall recall the precise definition in \S\ref{SectFundamental} below; essentially, fundamental logic is the weakest logic possessing introduction and elimination rules for $\wedge, \vee, \lnot$. Hermens \cite{He13} and Weingartner \cite{We10} have considered other notions of ``intuitionistic quantum logic.''\\

All these logics, however, while very natural, do not answer the motivating question, as they do not describe precisely the intersection of orthologic and intuitionistic logic. This fact perhaps explains why different natural combinations of orthologic and intuitionistic logic have been considered in the literature.
This differs from the situation of intuitionistic logic and linear logic: ``intuitionistic linear logic'' seems to be a standard notion (see e.g., Bierman \cite{Bi94}). From both the perspective of applications and that of ideology, this is a disadvantage, as in any context in which constructive quantum reasoning is required or expected, it is certainly desirable to have at one's disposal the strongest logic which is valid from both the constructive and quantum standpoints; thus we would like to know what this logic is.

In this article, we isolate the following remarkable axiom \textup{(Ex)}:
\begin{align*}
    \lnot \bigg[a \wedge \Big((b\wedge c) \vee (b &\wedge d)\Big)\bigg] \wedge a \wedge (c \vee e) \wedge \lnot\lnot f \vdash 
\lnot\lnot(a \wedge f) \\
    & \wedge \bigg[(a\wedge c) \vee (a \wedge e) \vee f\bigg] \wedge 
    \bigg[\Big( b \wedge (c \vee d)\Big) \vee \lnot \Big( b \wedge (c \vee d) \Big)\bigg].
\end{align*}
We prove that \textup{(Ex)} is valid both from the perspectives of intuitionistic logic and orthologic. We call \textit{Ex-logic} the result of adding axiom \textup{(Ex)} to fundamental logic. We prove:
\begin{theorem}\label{TheoremIntroEx}
The validities of Ex-logic are precisely the joint validities of orthologic and intuitionistic logic (in the signature $\{\vee,\wedge,\lnot\}$).
\end{theorem}
Theorem \ref{TheoremIntroEx} solves a problem of Holliday \cite{Ho23} and shows that \textup{(Ex)} precisely captures the joint truths of intuitionistic logic and orthologic (as the other axioms of fundamental logic merely express the meaning of the logical connectives).

It is somewhat remarkable that the common principles of orthologic and intuitionistic logic can be axiomatized, let alone so by a single axiom, in particular given that they include neither excluded middle nor distributivity and that the language of orthologic does not include implication; thus, there is no obvious way to combine multiple axioms into one. The axiom \textup{(Ex)} involves six variables. We conjecture that the joint validities of orthologic and the implication-free fragment of intuitionistic logic cannot be axiomatized by a single axiom with fewer than six variables, and that they cannot be axiomatized by any collection of axioms with fewer than four variables.

Our method of proof for Theorem \ref{TheoremIntroEx} is algebraic and quite general. Indeed, with some extra work, it adapts to yield 
Theorem \ref{TheoremMain} below, the main result of this article. 
It appears that this result gives a definite answer to the motivating question in the first paragraph, regardless of the precise logics that embody the philosophy of constructivism, as well as the precise logic of quantum reasoning
In the statement of Theorem \ref{TheoremMain}, ``axioms'' are expressions of the form $\phi_0 \vdash\phi_1$ as above.

\begin{theorem}\label{TheoremMain}
Let $\mathcal{L}_O$ be a logic extending orthologic by axioms $\{\phi_i\vdash\psi_i\}_i$ and $\mathcal{L}_I$ be a logic extending intuitionistic logic by axioms $\{\chi_j\vdash\theta_j\}_j$, all in the signature $\{\wedge, \vee,\lnot\}$. 
Then, the joint validities of $\mathcal{L}_O$ and $\mathcal{L}_I$ are axiomatized by adding to fundamental logic the following axioms:
\begin{enumerate}
\item  $\phi_i\vdash\lnot\lnot\psi_i$ for each $i$;
\item $\chi_j\vdash\theta_j\vee\lnot\theta_j$ for each $j$;
\item \textup{(Ex)}.
\end{enumerate}
\end{theorem}

We mention some particular cases as corollaries:

\begin{corollary}\label{CorollaryOML}
The joint validities of ortho-modular logic and the implication-free fragment of intuitionistic logic are axiomatized by fundamental logic augmented by the axiom  \textup{(Ex)} and
\[(a \vee \lnot a) \wedge (a \vee b) \vdash a \vee (\lnot a \wedge (a \vee b)).\]
\end{corollary}

\begin{corollary}\label{CorollaryDML}
The joint validities of orthologic and De Morgan logic are axiomatized by fundamental logic augmented by the axiom  \textup{(Ex)} and weak excluded middle.
\end{corollary}

The article is organized as follows. In \S\ref{SectFundamental} we review fundamental logic and its algebraic semantics. In \S\ref{SectExL} we study lattices validating axiom \textup{(Ex)} and establish their properties. In \S\ref{SectVEmb} we establish our main technical tool, an embedding theorem for Ex-lattices. In \S\ref{SectExLogic} we study Ex-logic and prove Theorem \ref{TheoremIntroEx}, as well a decomposition theorem separating \textup{(Ex)} into three independent axioms. In \S\ref{SectProofMain} we prove Theorem \ref{TheoremMain} and establish an isomorphism theorem between the lattice of super-Ex logics and the product of the lattices of superintuitionistic logics and superorthologics. In \S\ref{SectSubExL} we prove that there are infinitely many sub-Ex logics.

\section{Fundamental Logic} \label{SectFundamental}
Throughout, we work with propositional logic. Formulas are obtained from propositional variables and the constants $\top$ (truth) and $\bot$ (falsity) by means of the conjunctives $\wedge, \vee$, and $\lnot$. We emphasize that our signature throughout does not include implication, as this is not part of the signature of orthologic.

\begin{definition}
An \textit{introduction--elimination logic} is a binary relation $\vdash$ on propositional formulas such that for all formulas $\phi,\psi, \chi$, we have:
\begin{enumerate}
\item $\phi\vdash\phi$;
\item $\bot \vdash \phi$;
\item $\phi \vdash \top$;
\item $\phi\wedge\psi \vdash \phi$;
\item $\phi\wedge\psi\vdash \psi$;
\item $\phi\vdash \phi\vee\psi$;
\item $\psi\vdash \phi\vee\psi$;
\item $\phi\vdash\lnot\lnot\phi$;
\item $\phi\wedge\lnot\phi\vdash\bot$;
\item if $\phi\vdash\psi$ and $\psi\vdash\chi$, then $\phi\vdash\chi$;
\item if $\phi\vdash \psi$ and $\phi\vdash\chi$, then $\phi\vdash\psi\wedge\chi$; 
\item if $\phi\vdash\chi$ and $\psi\vdash \chi$, then $\phi\vee\psi\vdash \chi$;
\item if $\phi\vdash\psi$, then $\lnot\psi\vdash\lnot\phi$.
\end{enumerate}
\end{definition}

Fundamental logic was introduced by Holliday \cite{Ho23} and has a variety of equivalent definitions. The simplest definition is: fundamental logic is the smallest introduction--elimination logic. Thus, fundamental logic is able to introduce and use the logical connectives $\wedge,\vee,\lnot$, and it cannot do much more. From it, we can obtain orthologic by adding double-negation elimination, and we can obtain excluded middle by adding distributivity and the pseudo-complement law.

Fundamental logic can also be defined in terms of a Fitch-style natural deduction system (see \cite{Ho23}) and in terms of a sequent calculus (see \cite{AB24}). For our purposes, the algebraic semantics will be the most useful. In what follows, we assume some basic notions of universal algebra and lattice theory, such as the notion of the variety of ordered algebraic structures associated to a given logic $L$, which we denote by $\mathbb{V}(L)$ (intuitively, this is the universal class of all ordered structures on which the logic $L$ is valid). We refer the reader to \cite{Be15,Da02} for standard introductions on the topic.

\subsection{Fundamental lattices}
\label{SectFL}
\begin{definition}
A structure $(L, \wedge, \vee, 0, 1, \lnot)$ is said to be a \textit{fundamental lattice}  if $(L, \wedge, \vee, 0, 1)$ is a bounded lattice and $\lnot$ is a \textit{weak pseudo-complement} i.e., a unary operation for which the following hold:
\begin{enumerate}
\item \textit{antitonicity}, i.e., $a \leq b$ implies $\lnot b \leq \lnot a$ for all $a,b \in L$;
\item \textit{semi-complementation}, i.e., $a \wedge \lnot a = 0$ for all $a \in L$;
\item \textit{double-negation introduction}, i.e., $a \leq \lnot\lnot a$ for all $a \in L$.
\end{enumerate}
\end{definition} 

If this causes no confusion, we may avoid explicitly mentioning the operations on $L$ and simply refer to sets $L$ as lattices of one kind or another.
Given a weakly pseudo-complemented, bounded lattice $L$, a \textit{valuation} $v$ into $L$ assigns elements of $L$ to propositional variables. Valuations can be extended to arbitrary propositional formulas in the natural way. 
Often we abuse notation by omitting mention of $v$ if clear from the context, and identifying formulas with elements of $L$.
We say that an entailment $\varphi\vdash \psi$ is \textit{valid} in $L$, if 
\[\varphi \leq \psi\]
holds for all valuations into $L$.

\begin{proposition}[Holliday \cite{Ho23}]\label{PropositionFLComplete}
For formulas $\varphi,\psi$, the entailment $\varphi \vdash \psi$ is fundamentally true if and only if it holds in all fundamental lattices.
\end{proposition}

Holliday \cite{Ho23} observes that every bounded lattice can be equipped with a weak pseudo-complement. In general, there will be several ways of doing so.

Fundamental lattices can be refined into semantics for intuitionistic logic and orthologic. For orthologic, we must impose double-negation elimination, which also implies the law of excluded middle.

\begin{definition}
An \textit{ortholattice} is a fundamental lattice $(L, \wedge, \vee, 0, 1, \lnot)$ in which $\lnot$ satisfies double-negation elimination: $\lnot\lnot a = a$ for all $a \in L$.
\end{definition}

It is well known that orthologic is the logic of ortholattices, in the sense of Proposition \ref{PropositionFLComplete}; see Goldblatt \cite{Go74}. This can also be taken as the definition of orthologic.
For intuitionistic logic, we must impose distributivity and pseudo-complementation:

\begin{definition}
Let $(L, \wedge, \vee, 0, 1, \lnot)$ be a fundamental lattice. We say that $\lnot$ is a \textit{pseudo-complement} if for all $a,b\in L$, $a\wedge b = 0$ implies $b \leq \lnot a$. We say that $(L, \wedge, \vee, 0, 1, \lnot)$ is \textit{distributive} if conjunction distributes relative to disjunction and vice-versa. A \textit{Heyting lattice} is a pseudo-complemented, distributive fundamental lattice.
\end{definition}

The notion of a \textit{Heyting lattice} differs from the usual notion of a \textit{Heyting algebra}, which often involves implication. We consider implication-free signatures here as there is no natural notion of implication for ortholattices. Note that the reduced expressivity stemming from the lack of implication is offset by the explicit inclusion of entailment (as intuitionistic logic satisfies the deduction theorem). It is well-known that (the implication-free fragment of) intuitionistic logic is the logic of Heyting lattices, in the sense of Proposition \ref{PropositionFLComplete}. 

\begin{remark}\label{RemarkDeMorgan}
As noted by Holliday \cite{Ho23}, all the intuitionistically valid De Morgan laws are valid in all fundamental lattices. In particular, we have $\lnot (a \vee b) = \lnot a \wedge \lnot b$ and $\lnot a \vee \lnot b \leq \lnot (a \wedge b)$.
\end{remark}
Other properties of intuitionistic logic involving negation are true fundamentally, such as double-negation elimination for negated formulas.

\section{Ex-Lattices}\label{SectExL}

Our motivating question for this work is to study the joint truths of constructivism and quantum logic. We shall begin by equating these philosophies with intuitionistic logic and orthologic, respectively. 
In algebraic terms, our task amounts to investigating the inequalities valid in all ortholattices and in all Heyting lattices. The key to answer this question is to consider the following definition:

\begin{definition}
A \textit{Ex-lattice} is a fundamental lattice in which the following axiom
(Ex) is valid:
\begin{align*}
    \lnot \bigg[a &\wedge \Big((b\wedge c) \vee (b \wedge d)\Big)\bigg] \wedge a \wedge (c \vee e) \wedge \lnot\lnot f \leq 
\lnot\lnot(a \wedge f) \\
    & \wedge \bigg[(a\wedge c) \vee (a \wedge e) \vee f\bigg] \wedge 
    \bigg[\Big( b \wedge (c \vee d)\Big) \vee \lnot \Big( b \wedge (c \vee d) \Big)\bigg].
\end{align*}
\end{definition}

\begin{lemma}
    \textup{(Ex)} is valid in orthologic. Thus, every ortholattice is an Ex-lattice.
\end{lemma}
\begin{proof}
We prove that the following strengthening of \textup{(Ex)} is valid in all ortholattices
\begin{align*}
     a \wedge \lnot\lnot f \leq
\lnot\lnot(a \wedge f) \wedge f \wedge 
    \bigg[\Big( b \wedge (c \vee d)\Big) \vee \lnot \Big( b \wedge (c \vee d) \Big)\bigg].
\end{align*}
We need to show that $a \wedge \lnot\lnot f$ entails each of the three conjuncts in the consequent of the displayed axiom.
First, we have 
\[a \wedge \lnot\lnot f \leq a \wedge f \leq \lnot\lnot (a \wedge f)\]
by double-negation elimination. Similarly, we have
$\lnot\lnot f\leq f$.
Finally, the third conjunct is an instance of excluded middle and thus also valid.
\end{proof}

\begin{lemma}
    \textup{(Ex)} is valid in intuitionistic logic. Thus, every Heyting lattice is an Ex-lattice.
\end{lemma}
\begin{proof}
We prove that the following strengthening of \textup{(Ex)} is valid in all Heyting lattices:
\begin{align*}
    \lnot \bigg[a &\wedge \Big((b\wedge c) \vee (b \wedge d)\Big)\bigg] \wedge a \wedge (c \vee e) \wedge \lnot\lnot f \\
    &\leq 
\lnot\lnot(a \wedge f) \wedge \bigg[(a\wedge c) \vee (a \wedge e) \bigg] \wedge 
     \lnot \Big( b \wedge (c \vee d) \Big).
\end{align*}
Intuitionistically, we have
\begin{align}\label{eqVInt1}
    a \wedge \lnot\lnot f \leq \lnot\lnot a\wedge \lnot\lnot f \leq \lnot\lnot (a \wedge f).
\end{align}
By distributivity, we have 
\begin{align}\label{eqVInt2}
a \wedge (c \vee e) \leq (a \wedge c) \vee (a \wedge e) \leq (a \wedge c) \vee (a \wedge e).
\end{align}
Now, from the inequation $\neg(a \land p) \land a\wedge p \leq 0$ we obtain
\begin{align*}
    \lnot (a \wedge p) \wedge a \leq \lnot p.
\end{align*}
By instantiating $p = b \wedge (c \vee d)$, we obtain
\begin{align*}
    \lnot \Big(a \wedge b \wedge (c \vee d)\Big) \wedge a \leq \lnot \Big(b \wedge (c \vee d)\Big).
\end{align*}
which by distributivity implies the following:
\begin{align}\label{eqVInt3}
\lnot \bigg[a \wedge \Big((b\wedge c) \vee (b \wedge d)\Big)\bigg] \wedge a \leq \lnot \Big(b \wedge (c \vee d)\Big).
\end{align}
Putting together \eqref{eqVInt1}  \eqref{eqVInt2}, and \eqref{eqVInt3}, we  derive the strengthening of \textup{(Ex)} stated at the beginning of the proof.
\end{proof}

Thus, all ortholattices and all Heyting lattices are Ex-lattices. There is a technical converse to this fact, Theorem \ref{mainthm} in \S\ref{SectVEmb}, which will play a central role in our axiomatizations.

\subsection{Properties of Ex-lattices}

We now observe various consequences of the axiom (Ex).

\begin{lemma}\label{LemmaNu}
The following formula \textup{(Nu)}
\begin{align*}
\lnot\lnot p \wedge \lnot \lnot q \leq \lnot\lnot (p \wedge q)
\end{align*}
is valid in all Ex-lattices.
\end{lemma}
\proof
Substituting $c = e = 1$, $b = 0$, $f=q$, and  $a = \lnot\lnot p$ in \textup{(Ex)}, we have 
\[ \lnot\lnot p \wedge \lnot\lnot q \leq \lnot\lnot (\lnot\lnot p \wedge q).\]
Substituting $c = e = 1$, $b = 0$, $a = q$, and $f = p$ in \textup{(Ex)}, we have
\[\lnot\lnot p\wedge q \leq \lnot\lnot (p \wedge q).\]
Combining these two equations we have 
\[\lnot\lnot p \wedge \lnot\lnot q \leq \lnot\lnot(\lnot\lnot(p \wedge q)) \leq \lnot\lnot (p \wedge q),\]
as desired.
\endproof
The fact that \textup{(Nu)} is valid in all ortholattices and in all Heyting lattices was previously observed by Holliday \cite{Ho23}.

\begin{lemma}\label{LemmaVi}
The following formula \textup{(Vi)}
\begin{align*}
a \wedge (c \vee e) \wedge \neg \neg f \leq (a \wedge c) \vee (a \wedge e) \vee f
\end{align*}
is valid in all Ex-lattices.
\end{lemma}
\proof
Substituting $b = 0$ in \textup{(Ex)}, we have
\begin{align*}
a \wedge (c \vee e) \wedge \neg \neg f 
&\leq \lnot\lnot (a \wedge f) \wedge \big( (a \wedge c) \vee (a \wedge e) \vee f \big)\\
&\leq (a \wedge c) \vee (a \wedge e) \vee f,
\end{align*}
as desired.
\endproof

\begin{lemma}\label{LemmaCl}
The following formula \textup{(Cl)}
\begin{align*}
\neg \big( a \land ((b \land c) \lor (b \land d))\big) \land a \leq (b \land (c\lor d)) \lor \neg(b\land (c \lor d))
\end{align*}
is valid in all Ex-lattices.
\end{lemma}
\proof
Substituting $ e = f = 1$ in \textup{(Ex)}, we have
\[\neg \big(a \land ((b \land c) \lor (b \land d))\big) \land a \leq \neg \neg a \land \big((b \land (c\lor d)) \lor \neg(b\land (c \lor d))\big),\]

which clearly implies \textup{(Cl)}.
\endproof

\section{The Ex-Embedding Theorem}\label{SectVEmb}
In this section, we state and prove our main technical tool. For it, we need the notion of a \textit{fundamental homomorphism}. This is simply a homomorphism between fundamental lattices, i.e., an order-preserving function $e: L \to L'$ such that $e(0) = 0$, $e(1) = 1$, and $e$ commutes with the functions $\wedge, \vee$, and $\lnot$. Observe that this definition implies that the range $e[L]$ is a fundamental sublattice of $L'$.

\begin{theorem}[The Ex-Embedding Theorem] \label{mainthm}
    Let $(L, \land, \lor, 0,1, \neg)$ be an Ex-lattice. Then there are an ortholattice $O_L$, a Heyting lattice $A_L$, and a fundamental homomorphism 
\[e: L \to O_L \times A_L.\]
\end{theorem}
Here, the product of two fundamental lattices is defined the natural way, with 
\[(a, b) \leq (c,d) \text{ if and only if } a \leq c \text{ and } b \leq d.\]
This product can be seen as a fundamental lattice itself, setting 
\[\lnot (a, b) = (\lnot a, \lnot b).\]

The proof of \cref{mainthm} will require several lemmas. We start by constructing the ortholattice $O_L$ as a quotient of the original Ex-lattice $L$.

\begin{lemma}\label{LemmaResim}
    Let $L$ be an Ex-lattice. Then the relation $\resim \, \sset L \times L$ given by 
    \[a \resim b \leftrightarrow \neg a = \neg b\] 
    is a congruence relation.
\end{lemma}

\begin{proof}
    Clearly, $\neg a = \neg b$ implies that $\neg \neg a = \neg \neg b$. Moreover, since $\neg (a \lor b) = \neg a \land \neg b$ in any fundamental lattice (see Remark \ref{RemarkDeMorgan}), we have that $a \resim a'$ and $b \resim b'$ together imply that $\neg(a \lor b) = \neg a \land \neg b = \neg a' \land \neg b' = \neg (a' \lor b')$, whence $a\lor b \resim a' \lor b'$. Finally, observe that, for any $a,b \in L$:

    \[\neg(a \land b) = \neg\neg\neg(a \land b) = \neg(\neg \neg a \land \neg \neg b) = \neg \neg (\neg a \lor \neg b),\] where the second equality holds by Lemma \ref{LemmaNu}. But it follows from this that $a \resim a'$ and $b \resim b'$ together imply $a \land b \resim a' \land b'$.
\end{proof}

Given an Ex-lattice $L$ and $\resim$ defined as in Lemma \ref{LemmaResim}, let $a^*$ denote the equivalence class of $a$. We define 
\begin{equation}\label{eqDefO_L}
O_L = L / \resim = \{a^*: a \in L\}
\end{equation}
by putting $a^* \leq b^*$ if and only if there are $a_0 \in a^*, b_0 \in b^*$ such that $a_0 \leq b_0$.

\begin{lemma}
Suppose that $L$ is an Ex-lattice. Then, $O_L$ is an ortholattice and $a \mapsto a^*$ is a fundamental homomorphism of $L$ into $O_L$. 
\end{lemma}
\proof
By Lemma \ref{LemmaResim}, $\resim$ is a congruence relation, so the operations $\wedge$ and $\vee$ are well defined in $O_L$ and coincide with the lattice operations. It is easy to verify that it is a fundamental lattice using the congruence of $\resim$ relative to $\lnot$. The fact that it is an ortholattice follows immediately now, since $\neg \neg a \sim a$ for any $a \in L$.
\endproof

We have $a^*\neq b^*$ whenever $\lnot a\neq \lnot b$, but $O_L$ is unable to distinguish elements of $L$ with the same negation, by construction.
In order to embed $L$ into a product $O_L \times A_L$ as in the statement of the theorem, we need to define a Heyting lattice $A_L$ and a homomorphism $\phi: L \to A_L$ such that $\phi(a) \neq \phi(b)$ whenever $a \neq b$ and $\neg a = \neg b$. We do so now.

The following lemma refers to the axiom \textup{(Nu)} from Lemma \ref{LemmaNu}.
\begin{lemma} \label{int}
    Let $L$ be a fundamental lattice in which \textup{(Nu)} is valid, and let $a,b \in L$ be such that $\neg a = \neg b$. Then there is $c \in L$ such that $c \leq a, b \leq \neg \neg c$.
\end{lemma}

\begin{proof}
    Suppose that $\neg a = \neg b$. Let $c = a \land b$. Then we have that 
    \[a \leq \neg \neg a = \neg\neg a \land \neg \neg b = \neg \neg (a \land b).\]
    Similarly, we have $b \leq \neg \neg (a \land b)$ and hence $c \leq a, b \leq \neg \neg c$.
\end{proof}

Below, we recall that a (proper) \textit{filter} on a lattice $L$ is a nonempty subset of $L$ closed under meets $\wedge$, upwards closed under $\leq$, and not containing $0$. Below, we do not speak of improper filters.
Dually, a (proper) \textit{ideal} on a lattice $L$ is a nonempty subset of $L$ closed under joins $\vee$, downwards closed under $\leq$, and not containing $1$. A filter $F$ is \textit{prime} if $a \vee b \in F$ implies $a \in F$ or $b \in F$. Dually, an ideal $I$ is \textit{prime} if $a\wedge b \in I$ implies $a\in I$ or $b \in I$; equivalently, if the complement of $I$ is a filter. When dealing with Boolean algebras, maximal filters (and ideals) are always prime, but this is not necessarily true in the context of non-distributive algebras.

The following lemma refers to the axiom \textup{(Vi)} from Lemma \ref{LemmaVi}
\begin{lemma} \label{pft}
    Let $L$ be a fundamental lattice in which axiom \textup{(Vi)} is valid. Suppose $F$ is a filter on $L$ and $I$ is an ideal on $L$ such that $F \cap I = \emptyset$ and, for some $a \in L$, $a \in I$ and $\neg \neg a \in F$. Then there is a prime filter $P$ on $L$ such that $F \sset P$ and $P \cap I = \emptyset$.
\end{lemma}

\begin{proof}
    Let $F, I$ be respectively a filter and an ideal on $L$ such that $F \cap I =\emptyset$, and assume that there is $a \in L$ such that $a \in I$ and $\neg \neg a \in F$. By Zorn's Lemma, there is a filter $P$ maximal among filters on $L$ extending $F$ and disjoint from $I$. We claim that $P$ is prime. Suppose that $d \lor e \in P$ and, towards a contradiction, that neither of $d$ and $e$ belongs to $P$. As $d \not\in P$, then there is $c_1 \in P$ such that $c_1 \land d \in I$. To see this, observe that otherwise the set
    \[P' = \big\{x \in L : f \land d \leq x \text{ for some } f \in P\big\}\]
is a filter properly extending $P$ and disjoint from $I$, contradicting the maximality of $P$. Thus,  there is $c_1 \in P$ such that $c_1 \land d \in I$. Similarly, there is $c_2 \in P$ such that $c_2 \wedge e \in I$.

Let $c = c_1 \land c_2$. Since $\neg \neg a \in F$, we have that $\neg \neg a \land c \land (d \lor e) \in P$. Moreover, since $a \in I$, we have that $(c \land d) \lor (c \land e) \lor a \in I$. By \textup{(Vi)}, we have
\begin{align*}
\neg \neg a \land c \land (d \lor e) 
&\leq (c \wedge d) \vee (c\wedge e) \vee a,
\end{align*}
which contradicts the fact that $P \cap I = \emptyset$. Hence $P$ is prime.
\end{proof}

The following lemma refers to  the axiom \textup{(Cl)} from  \cref{LemmaCl}.
\begin{lemma} \label{neglma}
Let $L$ be a fundamental lattice in which \textup{(Cl)} is valid. Suppose that $P$ is a prime filter on $L$ such that $\neg a \notin P$ for some $a \in L$. Then there is a prime filter $Q \rset P$ such that $a \in Q$.
\end{lemma}

\begin{proof}
Suppose that $P$ is a prime filter and that $\neg a \notin P$. We distinguish two cases. If $a \lor \neg a \in P$, then $a \in P$ since $P$ is prime. Otherwise, we claim that $P \cup \{a\}$ generates a filter. We must show that for no $b \in P$ do we have $a \land b \leq 0$.
Suppose otherwise that there is some such $b \in P$, so $\lnot (a \wedge b) = 1$. By \textup{(Cl)}, letting $c = d = 1$ and swapping $a$ and $b$, we have $b \leq a \vee \lnot a$ contradicting the assumption that $a \vee \lnot a \not\in P$. Thus, there is no such $b$ and indeed $P \cup \{a\}$ generates a filter.

Using Zorn's lemma, extend $P \cup \{a\}$ to a maximal filter $Q$.
We claim that $Q$ is also prime.
We verify that the complement of $Q$ is an ideal. Suppose that $c \notin Q$ and $d \notin Q$ for some $c,d \in L$; we verify that $c \vee d \not \in Q$. 

Suppose first that for all $f\in Q$ we have $f \wedge c \neq 0$. Let 
\[Q' = \big\{x\in L: f \wedge c \leq x  \text{ for some } f \in Q \big\}.\]
Observe that $Q'$ is a filter, for if $f,f'\in Q$ are such that $f \wedge c \leq x$ and $f' \wedge c\leq y$, then
\[(f\wedge f') \wedge c \leq x\wedge y.\] 
Since $c \in Q'$, this is a proper extension of $Q$, contradicting its maximality. Thus, there is $f \in Q$ with $f \wedge c = 0$. Similarly, there is $g \in Q$ with $g \wedge d = 0$. Letting $b = f \wedge g$, we have $b \in Q$ and $b \land c = b \land d = 0$. This means that \[\neg ((b \land c) \lor (b \land d)) = \neg(b \land c) \land \neg (b \land d) = \neg 0 \land \neg 0 = 1.\]
Letting $a = 1$ in \textup{(Cl)}, we have 
\[1 \leq (b \land (c \lor d)) \lor \neg(b \land (c \lor d)).\] 
Since $P$ is prime, this means that either $b \land (c \lor d) \in P$ or $\neg(b \land (c\lor d)) \in P$. In the first case, we have that $c\lor d \in P$, which again by primality implies that either $c \in P\sset Q$ or $d \in P \sset Q$, a contradiction. Hence $\neg (b \land (c \lor d)) \in P \sset Q$. Since $b \in Q$, it follows that $c \lor d \notin Q$. This completes the proof that $Q$ is prime.
\end{proof}

Below, we let $P(L)$ be the set of prime filters on $L$. The structure $(P(L), \rset)$ is thus a partial order and can be regarded as a topological space in which the open sets are the downwards-closed subsets of $P(L)$, so that
the collection $Op(P(L))$ of open subsets of $P(L)$, ordered by inclusion, forms a complete Heyting lattice by defining negation as
\begin{align*}
\resim A &= \text{int}(P(L))\setminus A\\
& =P(L) \setminus \{Q \in P(L) \mid \exists P \in A: Q \sset P\}.
\end{align*}

We let $A_L$ be the Heyting sublattice of $Op(P(L))$ generated by sets of the form $\widehat{a} = \{P \in P(L) \mid a \in P\}$ for any $a \in L$.

\begin{lemma} \label{emblma}
Let $A_L$ be as above. The map
\begin{align*}
\widehat{\cdot}: L &\to A_L\\
a &\mapsto \widehat{a}
\end{align*}
is a surjective fundamental homomorphism.
Moreover, for any $a,b \in L$, if $a \nleq b$ and $\neg a = \neg b$, then $\widehat{a} \not \subseteq \widehat{b}$ and thus $\widehat{a} \not \leq \widehat{b}$.
\end{lemma}

\begin{proof}
    Clearly, $\widehat{\cdot}$ preserves $\wedge$ and $\vee$, since every $P \in P(L)$ is a prime filter. We claim that, for any $a \in L$, $\widehat{\neg a} = \resim \widehat{a}$. This amounts to showing that, for any $P \in P(L)$, $\neg a \in P$ if and only if $\forall Q \rset P: a\notin Q$. But the left-to-right direction is clear because any $Q \rset P$ is a filter, and the right-to-left direction follows from \cref{neglma}. Hence $\widehat{\cdot}$ is a fundamental homomorphism, which is surjective since $A_L$ is the sublattice of $Op(P(L))$ generated by sets of the form $\widehat{a}$ for some $a \in L$. Finally, suppose that $a \nleq b$, but $\neg a = \neg b$. By \cref{int}, there is $c \in L$ such that $c \leq a,b \leq \neg \neg c$. But this means that the filter $\upset{a}$ and the ideal $\dnset{b}$ satisfy the conditions of \cref{pft}. Hence there is $P \in P(L)$ such that $a \in P$ and $b \notin P$, which means that $\widehat{a} \nsubseteq \widehat{b}$.
\end{proof}

We can now complete the proof of \cref{mainthm}. 

\begin{proof}[Proof of the Ex-Embedding Theorem]
Given an Ex-lattice $L$, we let $O_L$ be the quotient ortholattice defined in \eqref{eqDefO_L} and we let $A_L$ the Heyting lattice just defined. We define 
\begin{align*}
e: L &\mapsto O_L\times A_L\\
a &\mapsto (a^*,\widehat{a}),
\end{align*}
where $a^*$ is as in the definition of $O_L$ and $\widehat{a}$ is as in Lemma \ref{emblma}.
Since both maps $a \mapsto a^*$ and $a \mapsto \widehat{a}$ are fundamental lattice homomorphisms, so is $e$. Moreover, $e$ is an embedding. Indeed, if $a \nleq b$ for some $a, b \in L$, then we distinguish two cases. If $\neg a \neq \neg b$, then $a^* \nleq b^*$ by definition of the quotient. If $\neg a = \neg b$, then $\widehat{a} \not\leq \widehat{b}$ by Lemma \ref{emblma}. Either way, it follows that $e(a) \nleq e(b)$, which concludes the proof.
\end{proof}

\section{Ex-Logic}\label{SectExLogic}
As an application of Theorem \ref{mainthm}, we obtain the characterization of the joint truths of orthologic and intuitionism. We show that they coincide precisely with fundamental logic augmented with the axiom \textup{(Ex)}, i.e., with \textit{Ex-logic}.

\begin{theorem}\label{TheoremOI}
Let $\varphi$, $\psi$ be formulas.
The following are equivalent:
\begin{enumerate}
\item $\varphi\vdash \psi$ is valid in all ortholattices and in all Heyting lattices.
\item $\varphi\vdash\psi$ is valid in all Ex-lattices.
\end{enumerate}
Thereore, the intersection of orthologic and the implication-free fragment of intuitionistic logic is precisely fundamental logic augmented with the axiom \textup{(Ex)}.
\end{theorem}

\begin{proof}
    Suppose that $\phi \vdash \psi$ is derivable in fundamental logic + \textup{(Ex)}. Then for any Ex-lattice $L$ and any valuation $v$ on $L$, $v(\phi) \leq v(\psi)$. In particular, this is true whenever $L$ is an ortholattice. Hence $\phi \vdash \psi$ is derivable in orthologic. Similarly, this is true whenever $L$ is a Heyting lattice, which means that $\phi \vdash \psi$ is derivable in intuitionistic logic. This shows that fundamental logic + \textup{(Ex)} is contained in the intersection of orthologic and intuitionistic logic. 
    
    For the converse inclusion, suppose that $\phi \vdash \psi$ is not derivable in fundamental logic + \textup{(Ex)}. Then there is an Ex-lattice $L$ and a valuation $c$ on $L$ such that $v(\phi) \nleq v(\psi)$. By the Ex-Embedding Theorem, there is an ortholattice $O_L$, a Heyting lattice $A_L$ and a fundamental embedding 
   \[e: L \to O_L \times A_L.\] Let $v^*$ be the valuation on $O_L \times A_L$ induced by letting $v^*(p) = e(v(p))$ for any propositional variable $p$. Clearly, $v^*$ induces two valuations $v^*_O$ and $v^*_A$ on $O_L$ and $A_L$ respectively, obtained by projecting onto the corresponding coordinate. Moreover, 
   \[(v^*_O(\phi),v^*_A(\phi)) = e(v(\phi)) \nleq e(v(\psi)) =(v^*_O(\psi),v^*_A(\psi)).\] This means that either $v^*_O(\phi) \nleq v^*_O(\psi)$ or $v^*_A(\phi) \nleq v^*_A(\psi)$. Hence either $\phi \vdash \psi$ is not provable in orthologic, or it is not provable in intuitionistic logic. This completes the proof.
\end{proof}

\begin{corollary}
    The validity problem for Ex-logic is in $\mathsf{co{-}NP}$.
\end{corollary}
\proof
This follows from Theorem \ref{TheoremOI}, together with the well-known facts that the validity problem for orthologic is in $\mathsf{P}$ (see e.g., Guilloud and Kun\v{c}ak \cite{GK24}) and the validity problem for Heyting lattices is in $\mathsf{co{-}NP}$ (see Shkatov and Van Alten \cite{ShVA21}).
\endproof 

\begin{theorem}
Over fundamental logic, axiom \textup{(Ex)} is equivalent to the conjunction of \textup{(Nu)}, \textup{(Vi)}, and \textup{(Cl)}. Therefore, \textup{(Nu)}, \textup{(Vi)}, and \textup{(Cl)} provide an alternate axiomatization of the intersection of orthologic and the implication-free fragment of intuitionistic logic. Moreover, this axiomatization is non-redundant: no two of the three axioms imply the other.
\end{theorem}
\proof 
To see that \textup{(Nu)}, \textup{(Vi)}, and \textup{(Cl)} are together equivalent to \textup{(Ex)} it suffices to see that they provide an alternate axiomatization of the intersection of orthologic and the implication-free fragment of intuitionistic logic.
To see that this is the case, observe that \textup{(Ex)} was not used directly in the proof of Theorem \ref{mainthm}; instead we only made use of the formulas \textup{(Nu)}, \textup{(Vi)}, and \textup{(Cl)}. Conversely, these formulas are all consequences of \textup{(Ex)}, by Lemma \ref{LemmaNu}, Lemma \ref{LemmaVi}, and Lemma \ref{LemmaCl}.

We now show that these formulas do not yield a redundant axiomatization by exhibiting fundamental lattices on which exactly two out of the three axioms \textup{(Nu)}, \textup{(Vi)} and \textup{(Cl)} are valid. In all the Hasse diagrams below, the fundamental negation is represented by a red dotted arrow. For the sake of clarity, the arrow from $0 \to 1$ and $1 \to 0$ is always omitted.

\begin{lemma}
Over fundamental logic, the formulas \textup{(Nu)} and \textup{(Vi)} together do not imply \textup{(Cl)}. 
\end{lemma}
\proof
Consider the following example:
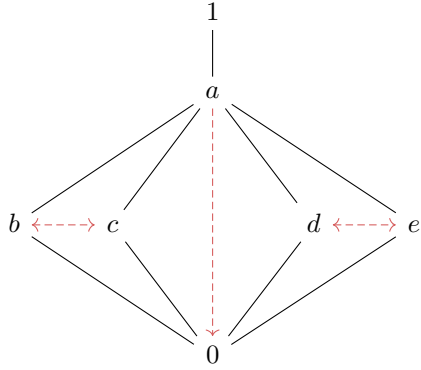
\[\begin{tikzcd}
	&& 1 \\
	&& a \\
	\\
	b & c && d & e \\
	\\
	&& 0
	\arrow[no head, from=2-3, to=1-3]
	\arrow[color={rgb,255:red,214;green,92;blue,92}, dashed, from=2-3, to=6-3]
	\arrow[no head, from=4-1, to=2-3]
	\arrow[color={rgb,255:red,214;green,92;blue,92}, dashed, tail reversed, from=4-1, to=4-2]
	\arrow[no head, from=4-2, to=2-3]
	\arrow[no head, from=4-4, to=2-3]
	\arrow[color={rgb,255:red,214;green,92;blue,92}, dashed, tail reversed, from=4-4, to=4-5]
	\arrow[no head, from=4-5, to=2-3]
	\arrow[no head, from=6-3, to=4-1]
	\arrow[no head, from=6-3, to=4-2]
	\arrow[no head, from=6-3, to=4-4]
	\arrow[no head, from=6-3, to=4-5]
\end{tikzcd}\]
We leave it to the reader to check that \textup{(Nu)} and \textup{(Vi)} are valid on this lattice. To see that \textup{(Cl)} is not valid, notice that we have 
\[\lnot ((b \wedge c) \vee (b \wedge d)) = 1\not\leq a = (b \wedge (c \vee d)) \vee \lnot (b\wedge (c\vee d)),\]
as desired.
\endproof

\begin{lemma}
Over fundamental logic, the formulas \textup{(Nu)}, and \textup{(Cl)} together do not imply \textup{(Vi)}.
\end{lemma}
\proof
Consider the following example.

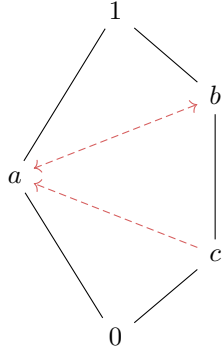
\[\begin{tikzcd}
	& 1 \\
	&& b \\
	a \\
	&& c \\
	& 0
	\arrow[no head, from=2-3, to=1-2]
	\arrow[no head, from=3-1, to=1-2]
	\arrow[color={rgb,255:red,214;green,92;blue,92}, dashed, tail reversed, from=3-1, to=2-3]
	\arrow[no head, from=4-3, to=2-3]
	\arrow[color={rgb,255:red,214;green,92;blue,92}, dashed, from=4-3, to=3-1]
	\arrow[no head, from=5-2, to=3-1]
	\arrow[no head, from=5-2, to=4-3]
\end{tikzcd}\]

We leave it to the reader to check that \textup{(Nu)} and \textup{(Cl)} are valid on this lattice. To check that \textup{(Vi)} is not valid, note that 
\[b \wedge (a \vee c) \wedge \neg \neg c = b \nleq c = (b \wedge a) \vee (b \wedge c) \vee c,\] as desired.
\endproof

\begin{lemma}
Over fundamental logic, the formulas \textup{(Vi)} and \textup{(Cl)} together do not imply \textup{(Nu)}. 
\end{lemma}
\proof
Consider the following example.

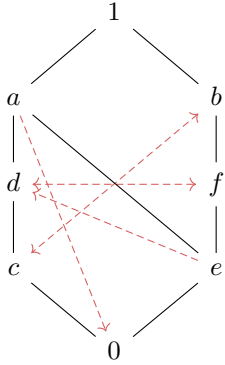
\[\begin{tikzcd}
	& 1 \\
	a && b \\
	d && f \\
	c && e \\
	& 0
	\arrow[no head, from=2-1, to=1-2]
	\arrow[color={rgb,255:red,214;green,92;blue,92}, dashed, from=2-1, to=5-2]
	\arrow[no head, from=2-3, to=1-2]
	\arrow[color={rgb,255:red,214;green,92;blue,92}, dashed, tail reversed, from=2-3, to=4-1]
	\arrow[no head, from=3-1, to=2-1]
	\arrow[color={rgb,255:red,214;green,92;blue,92}, dashed, tail reversed, from=3-1, to=3-3]
	\arrow[no head, from=3-3, to=2-3]
	\arrow[no head, from=4-1, to=3-1]
	\arrow[no head, from=4-3, to=2-1]
	\arrow[color={rgb,255:red,214;green,92;blue,92}, dashed, from=4-3, to=3-1]
	\arrow[no head, from=4-3, to=3-3]
	\arrow[no head, from=5-2, to=4-1]
	\arrow[no head, from=5-2, to=4-3]
\end{tikzcd}\]

We leave it to the reader to check that \textup{(Vi)} and \textup{(Cl)} are valid on this lattice. To check that \textup{(Nu)} is not valid, note that
\[\neg \neg a \land \neg \neg b = b \nleq f = \neg\neg(a \land b),\]
as desired.
\endproof
\noindent With these three counterexamples, the proof of the theorem is complete.
\endproof

\section{Super-Ex Logics}\label{SectProofMain}
\subsection{Proof of Theorem \ref{TheoremMain}}\label{SectProofMainProof}
We now prove Theorem \ref{TheoremMain}. The proof is a modification of the proof of Theorem \ref{TheoremOI}. 
The crucial ingredient in the proof of Theorem \ref{TheoremOI} was the Ex-Embedding Theorem; we will have to modify this theorem and its proof very slightly. 

Let $\mathcal{L}_O$ be a logic extending orthologic by axioms $\{\phi_i\vdash\psi_i\}_i$ and $\mathcal{L}_I$ be a logic extending intuitionistic logic by axioms $\{\chi_j\vdash\theta_j\}_j$, both in the signature $\{\wedge, \vee,\lnot\}$. 
\begin{claim}
For each $i$, the axiom $\phi_i \vdash \lnot\lnot\psi_i$ is valid in intuitionistic logic. Thus, it is a joint validity of $\mathcal{L}_I$ and $\mathcal{L}_O$.
\end{claim}
\begin{proof}
By assumption, $\phi_i\vdash \psi_i$ is classically valid. Using Glivenko's theorem, this means that $\phi_i \vdash \neg \neg \psi_i$ is derivable in intuitionistic logic. By double negation elimination in orthologic, we thus have that $\phi_i \vdash \neg \neg \psi_i$ is a joint validity of both $\mathcal{L}_O$ and $\mathcal{L}_I$.
\end{proof}

\begin{claim}\label{ClaimHeytingLattice}
Let $A$ be a Heyting lattice. Then for each $j$,
$\chi_j \vdash \theta_j$ is valid on $A$ if and only if $\chi_j \vdash \theta_j \vee \neg \theta_j$ is valid on $A$. Therefore, for each $j$, $\chi_j \vdash \theta_j \vee \lnot\theta_j$ is a joint validity of $\mathcal{L}_I$ and $\mathcal{L}_O$.
\end{claim}
\proof
Clearly, if $\chi_j \vdash \theta_j$ is valid on $A$, so is $\chi_j \vdash \theta_j \lor \neg \theta_j$. Now suppose that $\chi_j \vdash \theta_j \lor \neg \theta_j$ is valid on $A$. Fix a valuation $V$ on $A$. We will show that $V(\chi_j) \leq_A V(\theta_j)$. Note first that, since $\chi_j \vdash \theta_j$ is valid in $\mc{L}_I$ and $\mc{L}_I$ is a sublogic of classical logic, $\vdash \neg(\chi_j \land \neg \theta_j)$ is a classical validity. By Glivenko's theorem, this means that $\vdash \neg(\chi_j \land \neg \theta_j)$ is also an intuitionistic validity. Since $A$ is a Heyting lattice, we have that $V(\neg(\chi_j\land \neg \theta_j)) = 1$, which implies that $V(\chi_j \land \neg \theta_j) = 0$. Now we compute:
\begin{align*}
    V(\chi_j) = V(\chi_j) \land (V(\theta_j) \lor V(\neg\theta_j)) &= V(\chi_j \land \theta_j) \lor V(\chi_j \land \neg \theta_j) \\&= V(\chi_j \land \theta_j) \lor 0\\ &= V(\chi_j \land \theta_j),
\end{align*} which shows that $V(\chi_j) \leq_A V(\theta_j)$.

For the ``therefore'' part, simply notice that $\chi_j \vdash \theta_j\vee\lnot\theta_j$ is always valid in orthologic, by excluded middle.
\endproof

We let $\mathcal{L}$ be the result of extending fundamental logic by axioms $\{\phi_i\vdash\psi_i\}_i$, $\{\chi_j\vdash\theta_j\}_j$, and \textup{(Ex)}.
We have shown that $\mathcal{L}$ is contained in the intersection of $\mathcal{L}_O$ and $\mathcal{L}_I$. 
Let $L$ be an $\mathcal{L}$-lattice. The Ex-Embedding Theorem provides a fundamental embedding 
\[e: L \to O_L \times A_L\]
where $O_L$ is an ortholattice and $A_L$ is a Heyting lattice.
\begin{claim}\label{ClaimResimGen}
$O_L$ is an $\mathcal{L}_O$-lattice.
\end{claim}
\proof
Let $\phi_i \vdash \psi_i$ be any of the axioms of $\mathcal{L}_O$.
By Lemma \ref{LemmaResim}, $O_L$ is the quotient of $L$ via the congruence $\resim$, so if $\phi_i \vdash \neg \neg \psi_j$ is valid in $L$, then it is too in $O_L$. Since this is an ortholattice, it thus validates $\phi_i \vdash \psi_j$.
\endproof

\begin{claim}
$A_L$ is an $\mathcal{L}_I$-lattice. 
\end{claim}
\proof
Let $\chi_j \vdash \theta_j$ be any of the axioms of $\mathcal{L}_I$.
By \cref{emblma}, the map $a \mapsto \widehat{a}$ is a surjective fundamental homomorphism. Since $\chi_j \vdash \theta_j\vee\lnot \theta_j$ is valid on $L$, it is also valid on $A_L$. Since $A_L$ is a Heyting lattice, Claim \ref{ClaimHeytingLattice} implies that $\chi_j \vdash \theta_j$ is also valid on $A_L$.
\endproof

The argument in the proof of Theorem \ref{TheoremOI} now shows that $\mathcal{L}$ is contained in the intersection of $\mathcal{L}_O$ and $\mathcal{L}_I$, which completes the proof of the theorem.

\subsection{Orthomodular and De Morgan logics}
We now prove Corollary \ref{CorollaryOML} and Corollary \ref{CorollaryDML}. We restate them for convenience.
\begin{corollary}
The joint validities of ortho-modular logic and the implication-free fragment of intuitionistic logic are axiomatized by fundamental logic augmented by the axiom  \textup{(Ex)} and
\begin{equation}\label{eqOMAxiom}
(a \vee \lnot a) \wedge (a \vee b) \vdash a \vee (\lnot a \wedge (a \vee b))
\end{equation}
\end{corollary}
\proof
This corollary is not immediate from the statement of Theorem \ref{TheoremMain} (which would yield a different axiomatization), but follows from the proof.
It is well known that orthomodular logic is obtained from orthologic by adding the axiom 
\begin{equation}\label{eqOMAxiomEM}
(a \vee b) \vdash a \vee (\lnot a \wedge (a \vee b)).
\end{equation}
This is clearly equivalent to \eqref{eqOMAxiom} in the presence of excluded middle. It is easy to verify directly that \eqref{eqOMAxiom} is valid intuitionistically.

To establish the converse, we argue as in \S \ref{SectProofMainProof}, modifying the proof of Theorem \ref{TheoremMain}: let $\mathcal{L}$ be the result of extending fundamental logic by \eqref{eqOMAxiom} and use the Ex-Embedding Theorem to obtain a fundamental embedding 
\[e: L \to O_L \times A_L.\]
By Lemma \ref{LemmaResim}, \eqref{eqOMAxiom} is valid in $O_L$. But $O_L$ is an ortholattice and thus \eqref{eqOMAxiomEM} is valid in $O_L$, so $O_L$ is an orthomodular lattice. Using this observation in place of Claim \ref{ClaimResimGen}, the rest of the proof is as in \S \ref{SectProofMainProof}.
\endproof

\begin{corollary}
The joint validities of orthologic and De Morgan logic are axiomatized by fundamental logic augmented by the axiom  \textup{(Ex)} and weak excluded middle.
\end{corollary}
\proof
Immediate from Theorem \ref{TheoremMain} using the fact that
\[(\lnot p \vee\lnot\lnot p) \vee \lnot (\lnot p \vee\lnot\lnot p) = \lnot p \vee\lnot\lnot p\]
is fundamentally valid, which in turn is easy to verify directly.
\endproof

\subsection{A characterization of super-Ex logics}
In light of \cref{TheoremMain}, we can now associate to a pair $(\mc{L}_O, \mc{L}_I)$ of a quantum logic $\mc{L}_O$ and a superintuitionistic logic $\mc{L}_I$ the super-Ex logic $\mc{L}_O \cap\mc{L}_I$. Similarly, given a super-Ex logic $\mc{L}$, let $\mathbb{V}(\mc{L})$ be the variety of fundamental lattices determined by $\mc{L}$. We define $O_\mc{L} =\{ \phi \vdash \psi \mid \forall L \in \mathbb{V}(\mc{L}): O_L \models \phi \leq \psi \}$ and $I_\mathcal{L} = \{\chi \vdash \theta \mid \forall L \in \mathbb{V}(\mc{L}): A_L \models \chi \leq \theta\}$, where $O_L$ and $A_L$ are obtained from $L$ via the Ex-Embedding Theorem.

\begin{theorem} \label{prodthm}
    The maps: $(\mc{L}_O,\mc{L}_I) \mapsto \mc{L}_O \cap \mc{L}_I$ and $\mc{L} \mapsto (O_\mc{L},I_\mc{L})$ are inverses of each another, and determine an isomorphism between the lattice $\mathbf{SE}$ of super-\textup{Ex} logics and the lattice $\mathbf{SO}\times\mathbf{SI}$, where $\mathbf{SO}$ and $\mathbf{SI}$ are, respectively, the lattices of quantum and superintuitionistic logics (in the signature $\{\land, \lor, \neg\}$).
\end{theorem}

\proof
Clearly, both maps are order-preserving, so we only need to show that they are inverses of one another. Suppose $\mc{L}$ is a super-\textup{Ex} logic. Since $O_L$ and $A_L$ are both quotient lattices of $L$ for any \textup{Ex}-lattice $L$, it follows that $\mc{L} \sset O_\mc{L} \cap I_\mc{L}$. For the converse direction, suppose that $\phi\vdash \psi$ is not valid in $\mc{L}$. Then there is an \textup{Ex}-lattice $L$ such that $\phi\vdash \psi$ is not valid in $L$. By the Ex-Embedding Theorem, $L$ embeds into $O_L \times A_L$, so either $\phi\vdash \psi$ is not valid in one of $O_L$ or $A_L$. Either way, it follows that $\phi\vdash \psi$ does not belong to $O_\mc{L} \cap I_\mc{L}$.

Similarly, let us show that $O_{\mc{L}_O \cap \mc{L}_I} = \mc{L}_O$ and $I_{\mc{L}_O \cap \mc{L}_I} = \mc{L}_I$. Note first that $\mathbb{V}(\mc{L}_O) \cup \mathbb{V}(\mc{L}_I) \sset \mathbb{V}(\mc{L}_O \cap \mc{L}_I)$. Moreover, since $O_L = L$ whenever $L$ is an ortholattice and $A_L = L$ whenever $L$ is a Heyting lattice, this shows that $O_{\mc{L}_O \cap \mc{L}_I} \sset \mc{L}_O$ and $I_{\mc{L}_O \cap \mc{L}_I} \sset \mc{L}_I$. For the converse directions, suppose that $\phi \vdash \psi$ belongs to $\mc{L}_O$. Then $\phi \vdash \lnot\lnot\psi \in \mc{L}_O \cap \mc{L}_I$ by Theorem \ref{TheoremMain}. Since $\psi$ and $\lnot\lnot \psi$ are equivalent over orthologic, it follows that $\phi\vdash \psi$ is valid in $O_L$ for any $L \in \mathbb{V}(\mc{L}_O \cap \mc{L}_I)$, and therefore $\phi\vdash \psi$ belongs to $O_{\mc{L}_O}$. Similarly, if $\chi\vdash \theta$ belongs to $\mc{L}_I$, then $\chi \vdash \theta \vee \lnot\theta \in \mc{L}_O \cap \mc{L}_I$ by Theorem \ref{TheoremMain}. 
Since $\chi\vdash \theta$ and $\chi\vdash \theta\vee\lnot \theta$ are equivalent over intuitionistic logic, it follows that $\chi\vdash \theta$ is valid in $A_L$ for any $L \in \mathbb{V}(\mc{L}_O \cap \mc{L}_I)$, and therefore $\chi\vdash \theta$ belongs to $I_{\mc{L}_I}$. This completes the proof.
\endproof

\section{Sub-Ex Logics}\label{SectSubExL}
In this section, we study logics weaker than orthologic and intuitionistic logic, but in which \textup{(Ex)} is not valid. Our motivation is the following. Now that we have a clear picture of logics which support both constructive and quantum reasoning, we may wonder ``how far'' Holliday's fundamental logic is from such logics.
We prove:
\begin{theorem} \label{subexthm}
There exist infinitely many logics extending fundamental logic and contained in both orthologic and intuitionistic logic.
\end{theorem}
\proof
Dunn et al.~\cite{Detal05} have shown that the quantum logics associated to $2^n$-dimensional Hilbert spaces all differ. Let us recall their result: there exist formulas $\phi_m$ for $m\in\mathbb{N}$ such that $\phi_m$ has value $0$ in every distributive lattice.
%
%
Moreover, the inequality
\begin{equation}\label{eqAlphams}
    \phi_m \leq 0
\end{equation}
is valid in the ortholattice $Q(\mathbb{C}^{2^l})$ of closed subspaces of the Hilbert space $\mathbb{C}^{2^l}$ if and only if $l\leq m$. Consider the logic $\mathcal{F}_m$ obtained by extending fundamental logic by the axiom
\begin{equation}\label{eqAlphamInt}
    \vdash \phi_m \vee \lnot \phi_m.
\end{equation}
Then, $\mathcal{F}_m$ is a sub-Ex logic: that $\mathcal{F}_m$ is contained in intuitionistic logic follows from the fact that \eqref{eqAlphams} is a weakening of distributivity. That $\mathcal{F}_m$ is contained in orthologic follows from the fact that it is implied by excluded middle.

Given a natural number $l$, let $F_l$ be the lattice obtained by adding a new element on top of the lattice $Q(\mathbb{C}^{2^l})$. For convenience, we will denote as $1$ the top element of $F_l$ and as $1^*$ the top element of $Q(\mathbb{C}^{2^l})$, which is now the second largest element of $F_l$. Clearly, in order to turn $F_l$ into a fundamental lattice, it is enough to map $1$ to $0$ and to keep the orthocomplement in $Q(\mathbb{C}^{2^l})$, except for $\neg 0$, which is now $1$ instead of $1^*$. Note that there is a surjective homomorphism 
$h: F_l \to Q(\mathbb{C}^{2^l})$, which maps $1$ to $1^*$ and every element of $Q(\mathbb{C}^{2^l})$ to itself.

\begin{claim} \label{subexclaim1}
For any formula $\phi$, if $\phi$ is valid on $Q(\mathbb{C}^{2^l})$, then $\neg \neg \phi$ is valid on $F_l$.
\end{claim}
\proof
Note first that for any valuation $V$ on $F_l$, there is a valuation $V'$ on $Q(\mathbb{C}^{2^l})$ such that $V'(\phi) = h(V(\phi))$ for any formula $\phi$. Indeed, given a valuation $V$ on $F_l$, we can simply define $V'(p)$ as $h(V(p))$ for any propositional variable $p$. Now suppose that $\phi$ is valid on $Q(\mathbb{C}^{2^l})$. Then for any valuation $V$ on $F_l$, $h(V(\phi))=V'(\phi) = 1^*$. But this means that $V(\phi) \in \{1^*,1\}$. Either way, it follows that $V(\neg\phi) = 0$, and hence $V(\neg\neg\phi) = 1$.
\endproof

\begin{claim} \label{subexclaim2}
    For any valuation $V$ on $Q(\mathbb{C}^{2^l})$, there is a valuation $V'$ on $F_l$ such that $h(V'(\phi)) = V(\phi)$ for any formula $\phi$.
\end{claim}
\proof
    Given a valuation $V$ on $Q(\mathbb{C}^{2^l})$, we simply let $V'(p) = V(p)$ for any propositional letter $p$. Then a simple induction on the complexity of $\phi$ establishes the claim.
\endproof

\begin{claim}
The axiom \eqref{eqAlphamInt} is valid on $F_l$ if and only if $l \leq m$.
\end{claim}

\proof
Suppose that $l \leq m$. Then $\neg \alphad m$ is valid on $Q(\mathbb{C}^{2^l})$. By \cref{subexclaim1}, this means that $\neg \alphad m$ is also valid on $F_l$. For the converse direction, suppose that $m < l$. Then, by the result of Dunn et al.~\cite{Detal05}, there is a valuation $V$ on $Q(\mathbb{C}^{2^l})$ such that $0 < V(\alphad m) < 1^*$. By \cref{subexclaim2}, there is a valuation $V'$ on $F_l$ such that 
\[h(V'(\alphad m)) = V(\alphad m)\]
and 
\[h(V'(\neg \alphad m)) = V(\neg \alphad m).\]
Since $V(\alphad m) \neq 1^*$ and $V(\neg\alphad m) \neq 1^*$, it follows that \[V'(\alphad m) < 1^*\] and \[V'(\neg\alphad m) < 1^*.\] But this means that \[V'(\alphad m \lor \neg \alphad m) \leq 1^*,\] which shows that \eqref{eqAlphamInt} is not valid on $F_l$.
\endproof

It follows that the set $\{\mathcal{F}_m: m\in\mathbb{N}\}$ is an infinite family of logics extending fundamental logic and contained in both intuitionistic logic and orthologic, as desired. This completes the proof of the theorem.
\endproof

The proof just given in fact shows the following stronger fact:

\begin{corollary}
    There are infinitely many logics in the interval between $\textup{FL} + \textup{(Vi)} + \textup{(Nu)}$ and $\textup{FL} + \textup{(Ex)}$.
\end{corollary}

\begin{proof}
A closer inspection of the proof of \cref{subexthm} reveals that the fundamental lattice $F_l$ satisfies axioms $\textup{(Vi)}$ and $\textup{(Nu)}$. Consequently, the same proof shows that the logics $\{\mathcal{F}_m + \textup{(Vi)} + \textup{(Nu)} \mid m \in \mathbb{N}\}$ form an infinite family of logics between $\textup{FL} + \textup{(Vi)} + \textup{(Nu)}$ and $\textup{FL} + \textup{(Ex)}$.
\end{proof}

\begin{figure}[h]
\begin{center}
\[\begin{tikzcd}
	& {\textup{CL}} \\
	\\
	\\
	{\textup{IL}} && {\textup{OL}} \\
	\\
	\\
	& {\textup{Ex}} \\
	\\
	\\
	{\textup{FL + (Nu) + (Vi)}} & {\textup{FL + (Vi)+(Cl)}} & {\textup{FL + (Nu)+(Cl)}} \\
	\\
	\\
	{\textup{FL + (Vi)}} & {\textup{FL + (Nu)}} & {\textup{FL + (Cl)}} \\
	\\
	\\
	& {\textup{FL}}
	\arrow[""{name=0, anchor=center, inner sep=0}, Rightarrow, from=4-1, to=1-2]
	\arrow[""{name=1, anchor=center, inner sep=0}, Rightarrow, from=4-3, to=1-2]
	\arrow[""{name=2, anchor=center, inner sep=0}, Rightarrow, from=7-2, to=4-1]
	\arrow[""{name=3, anchor=center, inner sep=0}, Rightarrow, from=7-2, to=4-3]
	\arrow[Rightarrow, from=10-1, to=7-2]
	\arrow[from=10-2, to=7-2]
	\arrow[from=10-3, to=7-2]
	\arrow[from=13-1, to=10-1]
	\arrow[from=13-1, to=10-2]
	\arrow[from=13-2, to=10-1]
	\arrow[from=13-2, to=10-3]
	\arrow[from=13-3, to=10-2]
	\arrow[from=13-3, to=10-3]
	\arrow[from=16-2, to=13-1]
	\arrow[from=16-2, to=13-2]
	\arrow[from=16-2, to=13-3]
	\arrow[shorten <=27pt, shorten >=27pt, equals, from=2, to=1]
	\arrow[shorten <=27pt, shorten >=27pt, equals, from=3, to=0]
\end{tikzcd}\]
\end{center}
\caption{The structure of constructive quantum logics.}\label{FigurePictureEnd}
\end{figure}
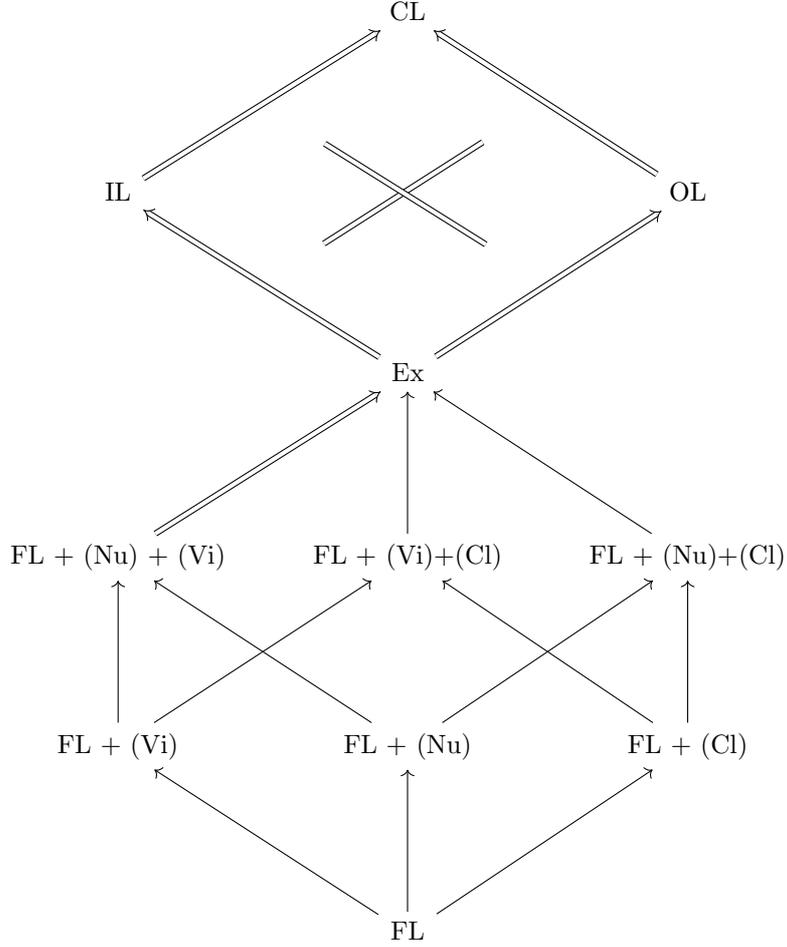

\section{Conclusions}

The work presented here gives a sufficiently clear picture of the structure of constructive quantum logics, which is summarized in Figure \ref{FigurePictureEnd}. In it, all introduction--elimination logics are depicted and Ex-logic represents a dividing line between the logics which support minimal quantum and constructive reasoning from those which do not. 

The crossed pair of double lines in the upper diamond of Figure  \ref{FigurePictureEnd} indicates the isomorphism given by Theorem \ref{prodthm}. Both the upper and lower diamonds contain infinitely many logics, and we use double arrows to indicate that it is known that the interval between two logics contains infinitely many logics. We conjecture that every arrow in the picture represents in fact infinitely many logics. We expect that there is much more to be said about constructive quantum logics, both in terms of theory and applications, and mention the following three questions as starting point:

\begin{question}
What are the joint validities of \textup{FL + (Nu)}, \textup{FL + (Vi)}, and \textup{FL + (Cl)}? What about the  joint validities of each pair of logics?
\end{question}

\begin{question}
What is the three-variable fragment of the intersection of intuitionistic logic and orthologic (in the signature $\{\wedge,\vee,\lnot\}$)?
\end{question}

\begin{question}
Is the validity problem for Ex-logic $\mathsf{co{-}NP}$-hard?
\end{question}

\subsection{Acknowledgements}
The authors would like to thank Rodrigo Almeida, Nick Bezhanishvili, Jan Byd\u{z}ovsk\'y, Mart\'in Soto, and Wesley Holliday for fruitful conversations. The work of the first author was partially supported by FWF grant STA-139.

\end{document}